\RequirePackage{etex}
\documentclass[12pt]{amsart}
\usepackage[utf8]{inputenc} 
\usepackage[active]{srcltx}
\usepackage{a4wide}
\usepackage{amsthm,amsfonts,amsmath,mathrsfs,amssymb}
\usepackage{dsfont}
\usepackage{mathtools}
\usepackage[T1]{fontenc}
\usepackage[utf8]{inputenc}
\usepackage{enumerate}
\usepackage{comment}
\usepackage[left=4cm,top=4cm,right=4cm,bottom=4cm]{geometry}
\numberwithin{equation}{section}

\usepackage{geometry}
\usepackage{graphicx}
\usepackage{amssymb}
\usepackage{amsmath}
\usepackage{amsthm}
\usepackage{listings}
\usepackage[colorlinks=true,urlcolor=blue,citecolor=blue,linkcolor=blue]{hyperref}
\usepackage{esint}
\usepackage{xcolor}

\usepackage{etex}
\usepackage[all]{xy}
\usepackage{pgf,tikz}
\usetikzlibrary{arrows}
\usetikzlibrary{patterns}
\usepackage{hyperref}

\newcommand{\R}{{\mathbb R}} 
\newcommand{\C}{{\mathbb C}} 
\newcommand{\N}{{\mathbb N}}

\newcommand{\T}{{\mathbb T}} 
\newcommand{\D}{{\mathbb D}}

\renewcommand{\Re}{\mathrm{Re}}

\RequirePackage{etex}

\newtheorem{theorem}{Theorem}[section]
\newtheorem{lemma}[theorem]{Lemma}

\theoremstyle{definition}
\newtheorem{definition}[theorem]{Definition}
\newtheorem{example}[theorem]{Example}

\theoremstyle{remark}
\newtheorem{remark}[theorem]{Remark}

\numberwithin{equation}{section}

\makeatletter
\@namedef{subjclassname@2020}{%
	\textup{2020} Mathematics Subject Classification}
\makeatother

\author[C. G\'omez-Cabello ]{Carlos G\'omez-Cabello}
\address{Departamento de Matem\'atica Aplicada II and IMUS, Escuela Politécnica Superior, Universidad de Sevilla, Calle Virgen de África, 7 41011 Sevilla, Spain}
\email{cgcabello@us.es}
\author[P. Lefevre]{Pascal Lefèvre}
\address{
Univ. Artois, UR 2462, Laboratoire de Mathématiques de Lens (LML), F-62 300 LENS,
FRANCE
}
\email{pascal.lefevre@univ-artois.fr}
\author[H. Queff\'elec]{Herv\'e Queff\'elec}
\address{Univ. Lille Nord de France, USTL, Laboratoire Paul Painlevé U.M.R. CNRS 8524, F-59 655
VILLENEUVE D’ASCQ Cedex, FRANCE
}
\email{herve.queffelec@univ-lille.fr}
\title[Integration type operators on Dirichlet series]{Integration type operators and point evaluation on weighted Bergman spaces of Dirichlet series.}



\date{\today}

\begin{document}

\maketitle
\begin{abstract}
The theory of Banach spaces of Dirichlet series has drawn an increasing attention in the recent 25 years. 
One of the main interest of this new theory is that of defining analogues of the classical spaces of analytic functions on the unit disc. 
In this sense, Bergman spaces were introduced  several years ago contributing to broaden the picture of this theory. In the paper presenting these new family of spaces, some of its most essential questions were considered. 
Among them, some partial estimates of the norm of the pointwise evaluation functional were given. 
In this work, we introduce a version of the Riemann-Liouville semigroup acting on these spaces, and, with this new tool, we are able to estimate the norm of this functional.
\end{abstract}

\section{Introduction}
The systematic study of Banach spaces of Dirichlet series began in the seminal paper \cite{seip}, where Hedenmalm, Lindqvist, and Seip introduced the Hilbert-Hardy space of Dirichlet series. This space consists in those Dirichlet series with square summable coefficients and gives rise to a Banach space of analytic functions in the maximal half-plane $\{s\in\C:\text{Re}(s)>1/2\}$. In the same work, the space $\mathcal{H}^{\infty}$ of bounded Dirichlet series in the right half-plane was also presented.

\ 
In 2002, Bayart \cite{bayart} extended the definition of the Hardy space $\mathcal{H}^2$ to the range $1\leq p<\infty$, which meant the introduction of a new family of Banach spaces of analytic functions: the $\mathcal{H}^p$-spaces. Little after, McCarthy \cite{mccarthy} was the first to consider Bergman spaces analogues in the setting of Dirichlet series, focusing its study in the Hilbert case. Several years later, in \cite{pascal}, Bailleul and Lefèvre 
embedded McCarthy's spaces in the larger family of Banach spaces of Dirichlet series  $\mathcal{A}^p_{\mu}$ where $\mu$ is a probability measure on $(0,\infty)$ such that $0\in\text{supp}(\mu)$. These spaces seem to be of significant interest since they have motivated several works such as \cite{maxime}, \cite{chen}, \cite{fu} or \cite{wang}. Let us remark that the main part
of these works studies 
a specific choice of the measure  $\mu$. More concretely, the family of density probability measures $d\mu_{\alpha}(\sigma)=h_{\alpha}(\sigma)d\sigma$, where 
\[
h_{\alpha}(\sigma)=\frac{2^{\alpha+1}}{\Gamma(\alpha+1)}\sigma^{\alpha}\,{\rm e}^{-2\sigma},\quad \alpha>-1.
\]
In such case, we shall write $\mathcal{A}^p_{\alpha}=\mathcal{A}^p_{\mu_{\alpha}}$. Some of the results stated in this work will refer to these specific spaces. We will write $\mathcal{A}^p_{\alpha,\infty}$ to denote the subspace of $\mathcal{A}^p_{\alpha}$ consisting of those Dirichlet series whose first coefficient equals zero. For a real number $\theta$, we shall denote $\C_{\theta}=\{s\in\C:\Re(s)>\theta\}$. For $\theta=0$, we write $\C_+=\C_0$.

\ 

One of the most natural questions arising when studying Banach spaces of Dirichlet series is that of computing or, at least, estimating the norm of the evaluation functionals. 
We recall that given a Banach space $X$ of analytic functions in a domain $\Omega$ in the complex plane $\C$ and a point $s\in\Omega$, the evaluation function at the point $s$, denoted by $\delta_s$, is defined as $\delta_s(f)=f(s)$, $f\in X$. 
For the case of the $\mathcal{H}^p$ spaces, $p\in[1,\infty)$, Bayart computed the norm of $\delta_s$, $s\in\C_{1/2}$, showing that $\|\delta_s\|_{(\mathcal{H}^p)^*}=\zeta(2\Re(s))$ (see \cite[Theorem 3]{bayart}). For the case of $\mathcal{A}^p_{\alpha}$-spaces, $p\in[1,\infty)$ and $s\in\C_{\frac{1}{2}}$, the following upper estimates are known (\cite[p.23]{pascal}):

$$\|\delta_s\|_{(\mathcal{A}^p_{\alpha,\infty})^*}\lesssim\frac{1}{\big({\rm Re}(s)-1/2\big)^{\frac{\alpha+2}{p}}}\quad\hbox{and}\quad\|\delta_s\|_{(\mathcal{A}^p_{\alpha})^*}\lesssim\Big(\frac{{\rm Re}(s)}{{\rm Re}(s)-1/2}\Big)^{\frac{\alpha+2}{p}}, \quad \alpha>-1\cdot$$

For $\alpha\in(-1,0)$, the lower estimate of the norm by a term equivalent to the right hand side still holds (\cite[p.26]{pascal}). 
However, for $\alpha\geq0$, the only known (until now) lower estimate was the following:
\begin{equation*}
\|\delta_s\|_{(\mathcal{A}^p_{\alpha})^*}\geq\frac{c}{({\rm Re}(s)-1/2)^{\frac{\alpha+2}{p}}|\log(2\sigma-1)|^\frac1p}\cdot
\end{equation*}
For the non-familiarised reader, let us point out that the consideration of point evaluation $\delta_s$ in the half-plane $\C_{\frac{1}{2}}$ is justified by the fact that $\mathcal{A}^p_{\mu}$ is a Banach space of analytic functions in this half-plane (\cite[Theorem 5]{pascal}).
This is also the case for the $\mathcal{H}^p$-spaces (\cite[p.207-208]{bayart}).

The organisation of the paper is the following. 
In Section \ref{operador} we present the Riemann-Liouville $I_t$ operator but acting on Dirichlet series. 
We also link the norm $\mathcal{A}^p_{\mu}$ of $f$, for $\mu$ a density measure, and the norm of $I_tf$. 
The introduction of this operator will allow us to tackle the norm estimation of the functional $\delta_s$. 
In Section \ref{integration}, we characterise the injection of $\mathcal{H}^p$ in $\mathcal{A}^p_{\alpha}$ in terms of the boundedness of the operator $I_t$.
The main result of this work is established in Section \ref{resultado}. 
In Theorem \ref{MAIN}, the following is proved for $1\leq p<\infty$
\begin{equation}\label{normeeval}
\forall\alpha>-1,\quad\|\delta_s\|_{(\mathcal{A}^p_{\alpha})^*}\approx\Big(\frac{{\rm Re}(s)}{{\rm Re}(s)-1/2}\Big)^{\frac{\alpha+2}{p}}\qquad \text{for every }s\in\C_{\frac{1}{2}}\;.
\end{equation}
This result does not only extend the already known lower estimates given in \cite[p.26]{pascal} from the range $\alpha\in(-1,0)$ to any value of $\alpha>-1$, but it also provides an affirmative answer to Question 2 from \cite{fu}.
The knowledge of the asymptotic behaviour of the point evaluation near the boundary on spaces of analytic functions is important and has many potential applications.
For instance, to mention at least one, it is one of the key ingredients in a forthcoming work on the Volterra operator \cite{GLQ2}.
Section \ref{introductory} is mainly expository and it is devoted to introduce all the background material which will be needed in the next sections.

\ 

We will use the notation $f(x)\lesssim g(x)$ if there is some constant $C>0$ such that $|f(x)|\leq |g(x)|$ for all $x$. If we have simultaneously that $f(x)\lesssim g(x)$ and that $g(x)\lesssim f(x)$, we write $f\approx g$.

\section{Dirichlet series and the spaces \texorpdfstring{$\mathcal{H}^p$}{Hp} and \texorpdfstring{$\mathcal{A}^p_{\mu}$}{A}}\label{espacios}\label{introductory}
\subsection{Dirichlet series}
We recall that a convergent Dirichlet series $\varphi$ (we write $\varphi\in\mathcal{D}$) is a formal series
$$
\varphi(s)=\sum_{n=1}^{\infty}a_nn^{-s},
$$
which converges somewhere, equivalently, in some half-plane $\C_{\theta}$ (see \cite[Lemma 4.1.1]{queffelecs}). When a Dirichlet series  $\varphi$ has only finitely many non-zero coefficients, we say that $\varphi$ is a Dirichlet polynomial. If $\varphi$ just has one non-zero coefficient, the series $\varphi$ is called a Dirichlet monomial. The first remarkable difference between Taylor series and Dirichlet series is that the latter ones converge in half-planes. In fact, there exist several {\sl abscissae of convergence} which, in general, do not coincide. The most remarkable ones are the following:  
$$
\sigma_{c}(\varphi)=\inf\{ \sigma:  \sum_{n=1}^{\infty}a_nn^{-s} \textrm{ is convergent on $\C_{\sigma}$}\};
$$
$$
\sigma_{u}(\varphi)=\inf\{ \sigma:  \sum_{n=1}^{\infty}a_nn^{-s} \textrm{ is uniformly convergent on $\C_{\sigma}$}\};
$$
$$
\sigma_{b}(\varphi)=\inf\{ \sigma:  \sum_{n=1}^{\infty}a_nn^{-s} \textrm{ has a bounded analytic extension to $\C_{\sigma}$} \};
$$
$$
\sigma_{a}(\varphi)=\inf\{ \sigma:  \sum_{n=1}^{\infty}a_nn^{-s} \textrm{ is absolutely convergent on $\C_{\sigma}$}\}.
$$

Let us mention that $\sigma_u=\sigma_b$ by Bohr's theorem (see \cite[Th.3.2.3]{queffelecs}).

The infinite polytorus $\T^{\infty}$ can be identified with the group of complex-valued characters $\chi$ on positive integers, more precisely arithmetic maps satisfying $|\chi(n)|=1$ for all $n\in\N$ and $\chi(mn)=\chi(m)\chi(n)$, for all $m,n\in\N$. Then, given $\chi\in\T^{\infty}$ and a Dirichlet series $f(s)=\sum_{n\geq1}a_nn^{-s}$ converging in some half-plane $\C_{\theta}$, the Dirichlet series
\[
f_{\chi}(s)=\sum_{n=1}^{\infty}a_n\chi(n)n^{-s}.
\]
is the normal limit in $\C_{\theta}$ of some sequence $\{f_{i\tau_k}\}_k$, where $\{\tau_k\}\in\R$.

We recall that, when $f$ is a Dirichlet series converging in some half place $\C_a$, then, for $s_0\in\overline{\C_0}$, we denote
\begin{equation*}
\forall s\in\C_a\,,\qquad f_{s_0}(s)=f(s+s_0).
\end{equation*}

\subsection{The \texorpdfstring{$\mathcal{H}^p$}{Hp} spaces }One of the key ingredients which enabled in \cite{seip} the introduction of the Hardy spaces of Dirichlet series is an observation due to H. Bohr. 
Essentially, Bohr observed that a Dirichlet series $f(s)=\sum_{n\geq1}a_nn^{-s}$ could be seen as a function on the infinite polydisc $\D^{\infty}$ by means of the fundamental theorem of arithmetic. 
Indeed, given $n\geq2$, it can be (uniquely) written as $n=p_1^{\alpha_1}\cdots p_r^{\alpha_r}$, where $\alpha_j\in\N\cup\{0\}$ for all $j$, and $\{p_j\}$ is the sequence of prime numbers. Then, letting $z_j=p_j^{-s}$, 
\[
f(s)=\sum_{n=1}^{\infty}a_n(p_1^{-s})^{\alpha_1}\cdots (p_r^{-s})^{\alpha_r}=\sum_{n=1}^{\infty}a_nz_1^{\alpha_1}\cdots z_r^{\alpha_r}.
\]
This allows us to write a Dirichlet series $f$ as a function $\mathcal{B}f$ in the infinite polydisc 
\[
\mathcal{B}f(z_1,z_2,\ldots)=\sum_{\substack{n\geq1\\ n=p_1^{\alpha_1}\cdots p_r^{\alpha_r}}}a_nz_1^{\alpha_1}\dots z_r^{\alpha_r}.
\]
Let $m_{\infty}$ be the normalised Haar measure on the infinite polytorus $\T^{\infty}$, that is, the product of the normalised Lebesgue measure of the torus $\T$ in each variable. Fix $p\in[1,\infty)$. We define the space $H^p(\T^{\infty})$ as the closure of the analytic polynomials in the $L^p(\T^{\infty},m_{\infty})$-norm. Now, given a Dirichlet polynomial $P$, its Bohr lift $\mathcal{B}P$ belongs to the space $H^p(\T^{\infty})$. This allows us to define its $\mathcal{H}^p$-norm as
\[
\|P\|_{\mathcal{H}^p}
=
\|\mathcal{B}P\|_{H^p(\T^{\infty})}.
\]
The $\mathcal{H}^p$-spaces of Dirichlet series are defined as the completion of the Dirichlet polynomials in the $\mathcal{H}^p$-norm.

Given a convergent Dirichlet series $f$ and $\sigma>0$, we denote by $f_{\sigma}$ 
the convergent Dirichlet series resulting from translating $f$ by $\sigma$: $f_{\sigma}(s)=f(\sigma+s)$. 
It is worth noting that the horizontal translation can be regarded as an operator acting on convergent Dirichlet series. In this case, we denote it by $T_{\sigma}$ so that $T_{\sigma}f=f_{\sigma}$. 
In fact, more general translation can be considered by taking $f_z$ with $\text{Re}(z)\geq0$, see for instance \cite[Theorem 11.20]{defant-peris}. 
It is noteworthy to  observe 
that the $\mathcal{H}^p$ spaces remain stable under horizontal translations to the right or vertical translations. 
\subsection{The \texorpdfstring{$\mathcal{A}^p_{\mu}$}{
	A}-spaces}
Let $P$ be a Dirichlet polynomial. Given a probability measure $\mu$ on $(0,\infty)$ such that $0$ belongs to $\text{supp}(\mu)$, we define
\begin{equation}\label{norm}
\|P\|_{\mathcal{A}^p_{\mu}}
:=
\left(
\int_0^{\infty}\|P_{\sigma}\|_{\mathcal{H}^p}^pd\mu(\sigma)
\right)^{\frac1p}.
\end{equation}
This clearly defines a norm. The spaces $\mathcal{A}^p_{\mu}$ are defined as the closure of the Dirichlet polynomials with respect to the norm \eqref{norm}. 
Taking this closure defines a family of Banach spaces of Dirichlet series in the half-plane $\C_{1/2}$ (\cite[Theorem 5]{pascal}).  
Let us mention that the condition $0\in\text{supp}(\mu)$ allows among other things (see e.g. \cite[Lemma 5.6 p.19]{GLQ2}) the consideration of the Dirac mass at $0$, which would give the space $\mathcal{H}^p$ as a limiting case. Here, we consider  the strictly larger class of spaces $\mathcal{A}^p_{\mu}$.

The norm \eqref{norm} can be written explicitly when $f$ is an $\mathcal{H}^p$ function. For a general $\mathcal{A}^p_{\mu}$-function, using the fact that $T_{\sigma}(\mathcal{A}^p_{\mu})\subset\mathcal{H}^p$ (\cite[Lemma 1]{pascal}, we can see that its norm can be computed as
\[
\|f\|_{\mathcal{A}^p_{\mu}}=\lim_{\tau\to0^+}\|f_{\tau}\|_{\mathcal{A}^{p}_{\mu}}.
\]

Let us mention the following remark comparing the point evaluation on spaces $\mathcal{A}^p_{\alpha}$ and $\mathcal{A}^p_{\alpha,\infty}$. 

\begin{lemma}\label{comparnormpointeval}
Let $p\geq1$, $\alpha>-1$ and $s\in\C_{\frac{1}{2}}$.
We have
$$\|\delta_s\|_{(\mathcal{A}^p_{\alpha,\infty})^*}\leq\|\delta_s\|_{(\mathcal{A}^p_{\alpha})^*}\leq1+2\|\delta_s\|_{(\mathcal{A}^p_{\alpha,\infty})^*}\;.$$
\end{lemma}
\begin{proof}
The first inequality is clear by restriction since  $\mathcal{A}^p_{\alpha,\infty}\subset\mathcal{A}^p_{\alpha}$.

For the second one, take any $f\in\mathcal{A}^p_{\alpha}$ and point out that $f-a_1\in\mathcal{A}^p_{\alpha,\infty}$ where $a_1$ is the constant Dirichlet coefficient.
Since $|a_1|\leq\|f\|_{\mathcal{A}^1_{\alpha}}\leq\|f\|_{\mathcal{A}^p_{\alpha}}$ (see \cite[Th.~9]{pascal}: there $w_1=1$), we have
$$|f(s)|\leq|a_1|+|f(s)-a_1|\leq\|f\|_{\mathcal{A}^p_{\alpha}}+\|\delta_s\|_{(\mathcal{A}^p_{\alpha,\infty})^*}\|f-a_1\|_{\mathcal{A}^p_{\alpha}} \leq\big(1+2\|\delta_s\|_{(\mathcal{A}^p_{\alpha,\infty})^*}\big)\|f\|_{\mathcal{A}^p_{\alpha}} $$
which gives the result.\end{proof}
Let us mention that a version of the preceding lemma in the more general framework $\mathcal{A}^p_{\mu}$ still holds true with the same proof.\medskip

The following three results (contained in \cite{pascal}) will be needed in the forthcoming sections.
\begin{theorem}\emph{(\cite[Corollary 1]{pascal})}\label{estimacionconocida}
Let $p\geq1$ and $\alpha>-1$. There exists a constant $c=c(\alpha,p)$ such that for every $s=\sigma+it$, $\sigma>1/2$,
\begin{equation*}
	\|\delta_s\|_{(\mathcal{A}^p_{\alpha,\infty})^*}\leq \frac{c}{\big(\sigma-1/2\big)^{\frac{\alpha+2}{p}}}\quad\hbox{and}\quad\|\delta_s\|_{(\mathcal{A}^p_{\alpha})^*}\leq c\Big(\frac{\sigma}{\sigma-1/2}\Big)^{\frac{\alpha+2}{p}}\cdot
\end{equation*}
\end{theorem}

\begin{theorem}\label{estiminf}
Let $p\geq1$, $\alpha\in(-1,0)$ and $a>\frac{1}{2}\,\cdot$ There exists a constant $c=c(\alpha,p,a)$ such that for every $s=\sigma+it$,
\begin{equation*}
	\sigma\in(1/2,a)\Longrightarrow\|\delta_s\|_{(\mathcal{A}^p_{\alpha})^*}\geq\|\delta_s\|_{(\mathcal{A}^p_{\alpha,\infty})^*}\geq\frac{c}{(\sigma-1/2)^{\frac{\alpha+2}{p}}}\cdot
\end{equation*}
\end{theorem}

\begin{proof}
We already know that $\displaystyle\|\delta_s\|_{(\mathcal{A}^p_{\alpha})^*}\geq\frac{c}{(\sigma-1/2)^{\frac{\alpha+2}{p}}}\cdot$
It is proved in \cite[p.26]{pascal}: see (i) when $p>1$ and (iii) for $p=1$.
This, together with Lemma~\ref{comparnormpointeval}, gives
$$\|\delta_s\|_{(\mathcal{A}^p_{\alpha,\infty})^*}\geq\frac{c}{2(\sigma-1/2)^{\frac{\alpha+2}{p}}}-\dfrac{1}{2}\geq\frac{c}{3(\sigma-1/2)^{\frac{\alpha+2}{p}}}$$
as soon as $\sigma$ is small enough, say less than $\sigma_0\in(1/2,a)$.

But when $\sigma\in(\sigma_0,a)$, we have, testing simply ${\rm e}_2(s)=2^{-s} $ (whose norm is $1$),
$$\|\delta_s\|_{(\mathcal{A}^p_{\alpha,\infty})^*}\geq2^{-\sigma}\geq2^{-a}\geq\frac{c'}{(\sigma-1/2)^{\frac{\alpha+2}{p}}}$$
for a suitable $c'>0$ not depending on $\sigma$.
We get the conclusion.\end{proof}

Let us remark that we cannot expect a uniform polynomial lower estimate for $\|\delta_s\|_{(\mathcal{A}^p_{\alpha,\infty})^*}$. 
Indeed one may check easily that $\|\delta_s\|_{(\mathcal{A}^p_{\alpha,\infty})^*}=O\big(2^{-\sigma}\big)$ when $\sigma\to+\infty$.

\begin{theorem}\emph{(\cite[Corollary 5]{pascal})}\label{estimacionzm}
Let $m\geq1$ be an integer and $\sigma>1/2$. Then, there exists a positive constant $c=c(m)$ such that
\[
\|\zeta^m(\sigma+\cdot)\|_{\mathcal{H}^2}\approx\frac{c}{(2\sigma-1)^{m^2/2}},\quad \text{when $\sigma\to\frac12$}.
\]
\end{theorem}
\section{The Dirichlet-Riemann-Liouville operator}\label{operador}
\begin{definition}
Let $f$ be an analytic function in the right half-plane $\C_+$, exponentially small when $\Re s\to \infty$. For $t>0$, we define the {\sl Riemann-Liouville operator} acting on $f$ as 
\begin{equation}\label{cete}
I_t(f)(s)=\frac{1}{\Gamma(t)}\int_s^{\infty}(u-s)^{t-1}f(u)du,\quad s\in\C_+,
\end{equation}
where $\Gamma(t)$ stands for the Gamma function at the point $t$. 
When the function $f$ happens to be a Dirichlet series, we shall refer to the operator as the {\sl Dirichlet-Riemann-Liouville operator}.
\end{definition}

Actually it is easy to check that the family of operators $(I_t)_{t>0}$ defines  an additive 
semi-group (The Riemann-Liouville semi-group on the right half line).

In the literature, the operator $I_t$ often  appears in the form given by \eqref{cete}. 
However, it will be more convenient for our purposes to rewrite  this operator 
as
\begin{align*}
I_t(f)(s)=\frac{1}{\Gamma(t)}
\int_0^{\infty}x^{t-1}f(x+s)dx.
\end{align*}
We are interested in studying the operator $I_t$ acting on Dirichlet series. As we are about to see, for the Dirichlet series to be stable under the action of the operator $I_t$, we must require their first coefficient to be zero. In particular, for $t=1$, the operator $I_t$ is the integration operator given by
\[
I(f)(s)=\int_s^{\infty}f(u)du.
\]
With this consideration, the operator $I_t$ acts on Dirichlet series, up to the sign, as a fractional integration operator.
\begin{lemma}\label{acciondir}
Let $f(s)=\sum_{n=2}^{\infty}a_nn^{-s}$ be a Dirichlet series convergent in $\C_+$.
Then,
\begin{equation}\label{expression}
I_t(f)(s)= \sum_{n=2}^{\infty}\frac{a_n}{(\log n)^t}n^{-s},\quad s\in\C_+, t>0.
\end{equation}
\end{lemma}
\begin{proof}
We first prove the result for Dirichlet monomials. Indeed, given $f(s)=n^{-s}$, $n\in\N$, $n\geq2$, we have that
\begin{align*}
I_t(f)(s)=\frac{1}{\Gamma(t)}\int_0^{\infty}u^{t-1}n^{-(s+u)}du
&=\frac{n^{-s}}{\Gamma(t)}\int_0^{\infty}u^{t-1}{\rm e}^{-u\log(n)}du\\
&=\frac{n^{-s}}{(\log n)^t}\cdot
\end{align*}
Clearly, by the linearity of the operator $I_t$, we have that \eqref{expression} holds for any Dirichlet polynomial $f$. Let now $f$ be a Dirichlet series with infinitely many non-zero terms. Then, by the analytic continuation, we have that \eqref{expression} still holds in the right half-plane $\C_+$, as desired. In particular, $\sigma_c(I_t(f))=\sigma_c(f)\leq0$.
\end{proof}

\begin{definition}
Let $h$ be a positive function in $(0,+\infty)$ such that $\|h\|_{L^1(\R_+)}=1$.
\begin{itemize}\leftmargin=0cm
\item[i)] We say that $h$ satisfies the $H$-condition if there exists a function $q:(0,\infty)\to(0,\infty)$  such that
\begin{equation}\label{condicionq}
	\int_0^{+\infty}\frac{x^{t-1}q(\frac{1}{x+1})^{1/p}}{(x+1)^{t+\frac1p}}dx<\infty,
\end{equation}
and that for almost every $\lambda\in(0,1)$,
\begin{equation*}
	h(\lambda u)\leq q(\lambda) h(u),\quad \text{a.e. on $(0,\infty)$}.
\end{equation*}
\item[ii)] We say that $h$ satisfies the $D$-condition if there 
exists $C>0$ such that
$$h(2u)\leq C h(u)\quad \text{for a.e. $u$ on $(0,\infty)$}.$$
\end{itemize}
\end{definition}

\begin{example}
The integrability hypothesis in the $H$-condition might look cumbersome. However, many natural choices for $q$ satisfy \eqref{condicionq}. That is the case of the positive real functions on $(0,\infty)$, $q(\lambda)=\lambda^{\alpha}$, $\alpha>-1$. For this choice of $q$ the integral in \eqref{condicionq} becomes 
\begin{align*}
C(t)=\int_0^{+\infty}\frac{x^{t-1}}{(x+1)^{t+\frac1p+\frac{\alpha}{p}}}dx.
\end{align*}
This integral converges for any $t>0$. Indeed, when $x\to\infty$,
\[
\frac{x^{t-1}}{(x+1)^{t+\frac1p+\frac{\alpha}{p}}}\sim\frac{1}{x^{1+\frac{1+\alpha}{p}}}
\]
and $1+(1+\alpha)/p>1$.
On the other hand, since $t>0$, the term $x^{t-1}$ guarantees the integrability in a neighbourhood of zero. 
\end{example}
We will need the following fractional version of the classical Cauchy
integral formula.
\begin{lemma}\label{identidadint}{~\newline}

\begin{enumerate}[(i)]
\item  Let $\displaystyle k_t=\int_{-\infty}^{+\infty}\frac{{\rm e}^{1-iy}}{(1-iy)^{t+1}}dy,\ t>0$.
We have $\displaystyle k_t=\frac{2\pi}{\Gamma(t+1)}\not=0.$\\

\item Let $f$ be a Dirichlet series convergent in $\C_+$. Then,
\begin{equation*}
	f(\theta)=\frac{1}{k_t}\int_{-\infty}^{+\infty}\frac{I_t(f)(i\tau)}{(\theta-i\tau)^{t+1}}d\tau,\quad \theta>0.
\end{equation*}
\end{enumerate}
\end{lemma}
\begin{proof} First we focus on $k_t$, which is clearly defined.
For $x\in\R$, let
$$g(x)=\frac{1}{\Gamma(t+1)} x^{t}\,{\rm e}^{-ax}1_{\R^+}(x),\quad a>0.$$
Obviously, $g\in L^1(\R)$ and, for every $y\in\R$, we compute its Fourier transform
$$\widehat{g}(y)=:\int_{\R} g(x){\rm e}^{-ixy}dx.$$
Applying the Cauchy formula to the function $z\mapsto F(z)=z^t\,{\rm e}^{-az}=\exp\big(-az+t\log z\big)$, holomorphic on $\C\setminus{\R^-}$ and continuous at $0$ (when defining $F(0)=0$), on a suitable path (depending on $y$) we get
$$\widehat{g}(y)=\frac{1}{(a+iy)^{t+1}}\cdot$$
 Since $\widehat{g}\in L^1(\R)$ as well, we get from the Fourier inversion formula applied at $x=1$:
$$\frac{{\rm e}^{-a}}{\Gamma(t+1)}=g(1)=\dfrac{1}{2\pi}\widehat{\widehat{g}}(-1)=\frac{1}{2\pi}\int_{\R} \frac{{\rm e}^{iy}}{(a+iy)^{t+1}}dy=\frac{1}{2\pi}\int_{\R} \frac{{\rm e}^{-iy}}{(a-iy)^{t+1}}dy.$$
Then, considering $a=1$ and rearranging we get $(i)$.

For part $(ii)$, let $f_m(s)=m^{-s}$, $m\in\N$, $m\geq2$, $s\in\C_+$. By the computations from Lemma \ref{acciondir}, we have that $I_t(g_m)(0)=(\log m)^{-t}$. Taking this into account, we find that
\begin{align*}
\int_{-\infty}^{+\infty}\frac{I_t(g_m)(i\tau)}{(\theta-i\tau)^{t+1}}d\tau
&=
\int_0^{+\infty}\frac{x^{t-1}}{\Gamma(t)}dx
\int_{-\infty}^{+\infty}\frac{f_m(x+i\tau)}{(\theta-i\tau)^{t+1}}d\tau 
\\&=\int_0^{+\infty}\frac{x^{t-1}m^{-x}}{\Gamma(t)}dx
\int_{-\infty}^{+\infty}\frac{{\rm e}^{-i\tau\log m}}{(\theta-i\tau)^{t+1}}d\tau \\
&=I_t(g_m)(0)\int_{-\infty}^{+\infty}\frac{{\rm e}^{-i\tau\log m}}{(\theta-i\tau)^{t+1}}d\tau .
\\&
=\int_{-\infty}^{+\infty}\frac{{\rm e}^{-iy}}{(\theta\log m\,-iy)^{t+1}}dy.
\end{align*}
Let now $a=\theta\log m$. Then, $f_m(\theta)=\Gamma(t+1)g(1)$.  In order to conclude, it suffices to argue as in part $(i)$ with the new choice of $a$ so that:
\[
\int_{-\infty}^{+\infty}\frac{{\rm e}^{-iy}}{(\theta\log m\,-iy)^{t+1}}dy
=
2\pi g(1)
=k_tf_m(\theta).
\]
We then conclude by linearity in a standard way.

\end{proof}

\begin{theorem}\label{equivnormas}
Let $d\mu(\sigma)=h(\sigma)d\sigma$ be a density probability measure on $\R_+$. Consider the positive measure given by $d\tilde{\mu}(u)=d\tilde{\mu}_{p,t}(u)=u^{pt}d\mu(u)$. Then,
\begin{enumerate}[(i)]
\item if $h$ satisfies the H-condition, there exists a constant $A>0$ such that
\begin{equation*}
	\forall f\in\mathcal{A}^p_{\widetilde{\mu}}\,,\qquad\|I_tf\|_{\mathcal{A}^p_{\mu}}\leq A\|f\|_{\mathcal{A}^p_{\widetilde{\mu}}}\;;
\end{equation*}
\item if $h$ satisfies the D-condition, 
there exists a positive constant $B>0$ such that
\begin{equation*}
	\forall f\in\mathcal{A}^p_{\mu}\,,\qquad\|f\|_{\mathcal{A}^p_{\widetilde{\mu}}}\leq B\|I_tf\|_{\mathcal{A}^p_{\mu}}\,.
\end{equation*}
\end{enumerate}
\end{theorem}

\begin{remark}
Point out that the measure $\widetilde{\mu}$ need not be a probability measure. Not even a finite measure. Nonetheless, we still have $0\in\text{supp}(\mu)$. So when we write $\mathcal{A}_{\widetilde{\mu}}^p$ the definition is the same as in the case of probability measures, in particular $\|\cdot\|_{\mathcal{A}^p_{\widetilde{\mu}}}$ is defined on Dirichlet polynomials as in \eqref{norm} (replacing $\mu$ by $\widetilde{\mu}$).
\end{remark}

\begin{proof} Let $f$ be a Dirichlet polynomial.

Notice first that, setting $u=\sigma x$, $\sigma>0$, we obtain that
\begin{align*}
(I_tf)_{\sigma}(s)=\frac{1}{\Gamma(t)}\int_0^{\infty}u^{t-1}f_{u+\sigma}(s)\,du=\frac{\sigma^t}{\Gamma(t)}\int_0^{\infty} x^{t-1}f_{\sigma(x+1)}(s)\,dx.
\end{align*}
Applying twice Minkowski's inequality (first relatively to $\|\cdot\|_{\mathcal{H}^p}$ and then relatively to the norm in $L^p(\R^+,hd\sigma$)) and the latter computation, we have that
\begin{align*}
\Gamma(t)\|I_tf\|_{\mathcal{A}^p_{\mu}}
&=\Bigg(\int_0^{\infty}\left\|\int_0^{\infty}\sigma^tx^{t-1} f_{\sigma(x+1)}\,dx\right\|_{\mathcal{H}^p}^ph(\sigma)\,d\sigma\Bigg)^{\frac{1}{p}}\\
&\leq\int_0^{\infty}
\left(
\int_0^{\infty}
\left\|\sigma^tx^{t-1}f_{\sigma(x+1)}\right\|_{\mathcal{H}^p}^ph(\sigma)\,d\sigma
\right)^{\frac1p}\,dx\\
&=\int_0^{\infty}x^{t-1}\left(
\int_0^{\infty}\sigma^{pt}\left\|f_{\sigma(x+1)}\right\|_{\mathcal{H}^p}^ph(\sigma)\,d\sigma
\right)^{\frac1p}\,dx.
\end{align*}   
Now, we carry out the change of variable $u=\sigma(x+1)$ and applying the $H$-condition, we get
\begin{align*}
\|I_tf\|_{\mathcal{A}^p_{\mu}}&\leq\int_0^{\infty}\frac{x^{t-1}}{\Gamma(t)}\left(
\int_0^{\infty}\sigma^{pt}\left\|f_{\sigma(x+1)}\right\|_{\mathcal{H}^p}^ph(\sigma)d\sigma
\right)^{\frac1p}dx\\
&
= \int_0^{\infty}\frac{x^{t-1}}{\Gamma(t)(x+1)^{t+\frac1p}}
\left(
\int_0^{\infty}u^{pt}\left\|f_{u}\right\|_{\mathcal{H}^p}^ph\left(\frac{u}{x+1}\right)\,du
\right)^{\frac1p}dx    \\
&\leq \int_0^{\infty}\frac{x^{t-1}q(\frac{1}{x+1})^{\frac1p}}{\Gamma(t)(x+1)^{t+\frac1p}}\left(
\int_0^{\infty}\left\|f_{u}\right\|_{\mathcal{H}^p}^pu^{pt}h(u)\,du
\right)^{\frac1p}dx\\
&= \dfrac{K}{\Gamma(t)}\|f\|_{\mathcal{A}^p_{\tilde{\mu}}}\,.
\end{align*}
where the value of $K$ is the value of the integral \eqref{condicionq}.

Finally, since the Dirichlet polynomials are dense in $\mathcal{A}^p_{\widetilde{\mu}}$ by definition, we obtain $(i)$.

To prove $(ii)$, we begin by observing that $I_t$ is translation invariant in the sense that we have $I_t(f_{\varepsilon})=\big(I_t(f)\big)_\varepsilon$, for any $\varepsilon>0$, and  $I_t(f_{\chi})=(I_t(f))_{\chi}$ for every character $\chi$. 
This observation together with an application of Lemma \ref{identidadint} at the point $\theta=\sigma/2$ yields
\begin{align*}
k_tf_{\chi}(\sigma)=k_t(f_{\chi})_{\frac{\sigma}{2}}(\sigma/2)=\int_{-\infty}^{+\infty}\frac{(I_tf)_{\chi}(\frac{\sigma}{2}+i\tau)}{(\frac{\sigma}{2}-i\tau)^{t+1}}d\tau\;.
\end{align*}
Taking the $L^p(\T^{\infty})$-norm on both sides of the latter identity and using Minkowski's inequality we obtain the following

\begin{align*}
k_t\left(
\int_{\T^{\infty}}|f_{\chi}(\sigma)|^pdm_{\infty}(\chi)
\right)^{\frac1p}
&\leq\int_{-\infty}^{+\infty}\left(
\int_{\T^{\infty}}|(I_tf)_{\chi}(\frac{\sigma}{2}+i\tau)|^pdm_{\infty}(\chi)
\right)^{\frac1p}\frac{d\tau}{|\frac{\sigma}{2}+i\tau|^{t+1}}\\
&=\int_{-\infty}^{+\infty}\|(I_tf)_{\frac{\sigma}{2}}\|_{\mathcal{H}^p}\frac{d\tau}{(\frac{\sigma^2}{4}+\tau^2)^{\frac{t+1}{2}}}\cdot
\end{align*}
With a change of variable and recalling that $t>0$, we find that
$$k_t \left(\int_{\T^{\infty}}|f_{\chi}(\sigma)|^pdm_{\infty}(\chi)\right)^{\frac1p}
\leq \frac{\|(I_tf)_{\frac{\sigma}{2}}\|_{\mathcal{H}^p}}{\sigma^t}\int_{-\infty}^{+\infty}\frac{1}{(\frac14+x^2)^{\frac{t+1}{2}}}dx
\leq C(t)\frac{\|(I_tf)_{\frac{\sigma}{2}}\|_{\mathcal{H}^p}}{\sigma^t}$$
where $C(t)$ is a constant since $t>0$. We conclude by integrating the $p$-th power on both sides of the latter chain of inequalities with respect to the measure $d\tilde{\mu}(\sigma)=\sigma^{pt}d\mu(\sigma)$ to obtain
\begin{align*}
k_t^p \|f\|_{\mathcal{A}^p_{\widetilde{\mu}}}^p
&\leq C(t)^p\int_0^{\infty}\|(I_tf)_{\frac{\sigma}{2}}\|_{\mathcal{H}^p}^ph(\sigma)d\sigma=2C(t)^p\int_0^{\infty}\|(I_tf)_{u}\|_{\mathcal{H}^p}^ph(2u)du\\&
\leq C_1(t)\int_0^{\infty}\|(I_tf)_{u}\|_{\mathcal{H}^p}^ph(u)du\approx\|I_tf\|_{\mathcal{A}^p_{\mu}}^p
\end{align*}
where in the last inequality we have made use of the $D$-condition.

Finally, since the Dirichlet polynomials are dense in $\mathcal{A}^p_{\mu}$ by definition, we obtain $(ii)$.\end{proof}

\begin{remark}
Notice that the family of densities $d\mu_{\alpha}$ do not satisfy the D-condition. 
Nonetheless, statement $(ii)$ still holds for these densities. 
We can prove a theorem analogue to Theorem \ref{equivnormas} 
and instead of using the $D$-condition, make use of the following Lemma. In \cite[Theorem 3.8]{fu}, a similar result was proven for a specific choice of measures and only for the integration operator {\sl i.e.} in the case $t=1$.
\end{remark}

\begin{lemma}\label{lemmnormequiv}
Let $\beta>-1$ and $1\leq p<\infty$. 

For every Dirichlet polynomial $f$ whose first coefficient equals zero, we have
\begin{equation*}
\int_0^{+\infty}\|f_u\|_{\mathcal{H}^p}^pu^{\beta}du\approx \|f\|_{\mathcal{A}^p_{\mu_{\beta}}}^p
\end{equation*}
\end{lemma}
\begin{proof}
We only prove that $\displaystyle \int_0^{+\infty}\|f_u\|_{\mathcal{H}^p}^pu^{\beta}du\lesssim\|f\|_{\mathcal{A}^p_{\mu_{\beta}}}^p$ since the reverse inequality is obvious.

Take $\delta\in(0,1)$ such that $p(\ln2)\frac{1-\delta}{\delta}\geq2$. 
By \cite[Theorem 8.1]{BaQS} (see the proof p.581-582), we have 

\begin{align*}
\int_0^{+\infty}\|f_u\|_{\mathcal{H}^p}^pu^{\beta}du   &=
\int_0^{+\infty}\|f_{\delta u+(1-\delta)u}\|_{\mathcal{H}^p}^pu^{\beta}du
\lesssim
\int_0^{+\infty}\|f_{\delta u}\|_{\mathcal{H}^p}^p2^{-(1-\delta)up}u^{\beta}du\\
&=c_{\delta}\int_0^{+\infty}\|f_{v}\|_{\mathcal{H}^p}^p{\rm e}^{-p(\ln2)\frac{1-\delta}{\delta}v}v^{\beta} dv\\
&\lesssim
\int_0^{+\infty}\|f_v\|_{\mathcal{H}^p}^p{\rm e}^{-2v}v^{\beta}dv=\|f\|_{\mathcal{A}^p_{\beta}}^p.
\end{align*}
\end{proof}

As a consequence of Theorem \ref{equivnormas} and Lemma \ref{lemmnormequiv}, we obtain immediately the useful
\begin{theorem}\label{especials}
For every $p\geq1$, $\alpha>-1$, $t>0$ and $f\in\mathcal{D}$ whose first coefficient equals zero. 
We have $f\in\mathcal{A}^p_{\alpha+tp}$ if and only if $I_t(f)\in\mathcal{A}^p_{\alpha}$, and
\[
\|f\|_{\mathcal{A}^p_{\alpha+tp}}\approx\|I_tf\|_{\mathcal{A}^p_{\alpha}}.
\]
\end{theorem}

Let us mention the following particular case: for a Dirichlet polynomial $P$ whose first coefficient equals zero, and every integer $m\geq1$, we have
$I_m(P^{(m)})=P$, therefore 
$$\|P^{(m)}\|_{\mathcal{A}^p_{\alpha+mp}}\approx\|P\|_{\mathcal{A}^p_{\alpha}}$$
with underlying constants depending on $p,m,\alpha$ only.

\section{A word on the integration operator}\label{integration}
The following result provides a characterisation of the injection between $\mathcal{H}^p$ and some $\mathcal{A}^p_{\alpha}$-type spaces in terms of the boundedness of the operator $I_t$. 
Some of the implications were known, but the theorem provides a bigger picture of the boundedness of $I_t$ and we think it is interesting in its own.
\begin{theorem}
Let $1\leq p,q <\infty$, $\alpha>-1$, and $t>0$. Then, the following statements are equivalent:
\begin{enumerate}[(i)]
\item $I_t:\mathcal{H}^p\to \mathcal{A}^q_{\alpha}$ is bounded.
\item $\emph{Id}:\mathcal{H}^p\to \mathcal{A}^q_{\alpha+qt}$ is bounded.
\item $p\geq q$.
\end{enumerate}
\end{theorem}
\begin{proof}
$(iii)\implies(i)$ is clear thanks to Theorem \ref{especials}. 

Let us assume $(i)$ and that $(iii)$ does not hold {\sl i.e.} $p<q$. In fact, we can assume that $p$ and $q$ are dyadic rationals. Hence, there exist infinitely many integers $m$ such that $mp/2$ and $mq/2$ are positive integers. Taking this into account, since $(i)$ holds:
\begin{align*}
\|\zeta^m_{\sigma}\|_{\mathcal{H}^p}^q
\geq
\|I_t(\zeta^m_{\sigma})\|_{\mathcal{A}^q_{\alpha}}^q
\approx
\int_0^{\infty}\|(\zeta_{\sigma}^m)_u\|_{\mathcal{H}^q}^qd\mu_{\alpha+qt}(u)
&=\int_0^{\infty}\|\zeta_{\sigma+u}^{mq/2}\|_{\mathcal{H}^2}^2d\mu_{\alpha+qt}(u)
\\
&
\geq\int_0^{2\sigma-1}\|\zeta_{\sigma+u}^{mq/2}\|_{\mathcal{H}^2}^2d\mu_{\alpha+qt}(u)
\end{align*}
Now, thanks to Theorem \ref{estimacionzm} and since ${\rm e}^{-2u}\approx1$ on $[0,2\sigma-1]$ when  $\sigma\to1/2^+$
we find that 
\begin{align*}
\int_0^{2\sigma-1}\|\zeta_{\sigma+u}^{mq/2}\|_{\mathcal{H}^2}^2d\mu_{\alpha+qt}(u)
&\approx
\int_0^{2\sigma-1}\frac{1}{(2\sigma+2u-1)^{\frac{m^2q^2}{4}}}u^{\alpha+qt}du\\&
\geq
\int_0^{2\sigma-1}\frac{1}{\big(3(2\sigma-1)\big)^{\frac{m^2q^2}{4}}}u^{\alpha+qt}du
\approx
\frac{1}{(2\sigma-1)^{\frac{m^2q^2}{4}-\alpha-qt-1}}\cdot
\end{align*}
On the other hand, by  Theorem \ref{estimacionzm}, we have that
\begin{align*}
\|\zeta_{\sigma}^m\|_{\mathcal{H}^p}^p
\approx\|\zeta^{\frac{mp}{2}}_{\sigma}\|_{\mathcal{H}^2}^2
\approx
\frac{1}{(2\sigma-1)^\frac{m^2p^2}{4}}\cdot
\end{align*}
Hence, still having in mind that $\sigma\to1/2^+$, we should have
\[
\frac{m^2pq}{4}\geq\frac{m^2q^2}{4}-(qt+1+\alpha).
\]
That is, $$p\geq q-4(qt+1+\alpha)/qm^2,$$
which is false when $m$ is large enough. So $(i)\implies(iii)$.

It remains to prove the equivalence between $(ii)$ and $(iii)$. 
Suppose first that Id$:\mathcal{H}^p\to\mathcal{A}^q_{\alpha+qt}$ is bounded. Using this and Theorem \ref{especials}, we have that
$$\forall f\in \mathcal{H}^p\,,\qquad\|f\|_{\mathcal{H}^p} \gtrsim\|f\|_{\mathcal{A}^q_{\alpha+qt}}\approx 
\|I_t(f)\|_{\mathcal{A}^q_{\alpha}}.$$
This implies $(i)$, which is equivalent to $(iii)$, and the conclusion follows. 

Let us now assume that $(iii)$ holds. Then, by $(i)$, the Riemann-Liouville integration operator $I_t$ maps boundedly $\mathcal{H}^p$ into $\mathcal{A}^q_{\alpha}$. Hence, using  Theorem \ref{especials},
$$\forall f\in \mathcal{H}^p\,,\qquad\|f\|_{\mathcal{H}^p}\gtrsim \|I_t(f)\|_{\mathcal{A}^q_{\alpha}}\approx\|f\|_{\mathcal{A}^q_{\alpha+qt}}.$$\end{proof}

In particular, we have the following statement for the integration operator:

Let $1\leq p,q <\infty$, $\alpha>-1$. Then, the following assertions are equivalent:
\begin{enumerate}
\item $I:\mathcal{H}^p\to \mathcal{A}^q_{\alpha}$ is bounded.
\item $\emph{Id}:\mathcal{H}^p\to \mathcal{A}^q_{\alpha+q}$ is bounded.
\item $p\geq q$.
\end{enumerate}

\begin{remark}
As a consequence of a theorem due to Bayart, see \cite[p. 50]{bayart1}, one has that the integration operator $I$ does not map $\mathcal{H}^1$ into $\mathcal{H}^2$ (our statement here is stronger). 
In fact, using Bayart's result and the coefficient estimates from \cite[Theorem 9]{pascal} some cases of the equivalence between $1)$ and $3)$ can also be deduced. 
The equivalence between 3) and 1) was proven in \cite[Theorem 3.1]{chen} for $p=q$. On the other hand, the equivalence between 2) and 3) was established in \cite[Corollary 4]{pascal} for the spaces $\mathcal{A}^q_{\alpha}$ with $\alpha=0$. Fu, Guo, and Yan using a different argument proved in \cite[Lemma 5.7]{fu} the compactness of the integration operator $I$ on $\mathcal{H}^p$ and $\mathcal{A}^p_{\alpha}$ for every $p\geq1$ and $\alpha>-1$.
\end{remark}

\section{Main result}\label{resultado}
\begin{lemma}\label{decreasing}
Let $u\in\C_{1/2}$. The functions $\Re(u)\mapsto\|\delta_u\|_{(\mathcal{A}^p_{\alpha})^*}$ and $\Re(u)\mapsto\|\delta_u\|_{(\mathcal{A}^p_{\alpha,\infty})^*}$ are non-increasing.
\end{lemma}
\begin{proof}
Let $u,v\in\C_{1/2}$ be such that $\Re (u)>\Re (v)$. Then, by \cite[Proposition 11.20]{defant-peris},
\begin{align*}
|f(u)|=|f_{u-v}(v)|\leq \|f_{u-v}\|_{\mathcal{A}^p_{\mu}}\|\delta_{v}\|_{(\mathcal{A}^p_{\mu})^*}\leq \|f\|_{\mathcal{A}^p_{\mu}}\|\delta_v\|_{(\mathcal{A}^p_{\mu})^*}.
\end{align*}
Then, $\|\delta_u\|_{(\mathcal{A}^p_{\mu})^*}\leq \|\delta_v\|_{(\mathcal{A}^p_{\mu})^*}$, and the conclusion follows for $\mathcal{A}^p_{\alpha}$.
The proof is the same for $\mathcal{A}^p_{\alpha,\infty}$.\end{proof}

We are now ready to establish the main result of this work.
\begin{theorem}\label{MAIN}
Let $p\geq1$, $\alpha>-1$ and $a>\frac{1}{2}\,\cdot$
Then
\begin{equation}\label{encadrinfini}
\|\delta_s\|_{(\mathcal{A}^p_{\alpha,\infty})^*}\approx\frac{1}{(\sigma-1/2)^{\frac{\alpha+2}{p}}} \qquad\text{for every }\sigma\in(1/2,a)\,,
\end{equation}

\begin{equation}\label{encadr}
\|\delta_s\|_{(\mathcal{A}^p_{\alpha})^*}\approx\Big(\frac{\sigma}{\sigma-1/2}\Big)^{\frac{\alpha+2}{p}} \qquad\text{for every }\sigma>\dfrac{1}{2}\,,
\end{equation}

In particular, 
\begin{equation}\label{mainapprox}
\|\delta_s\|_{(\mathcal{A}^p_{\alpha,\infty})^*}\approx\|\delta_s\|_{(\mathcal{A}^p_{\alpha})^*}\approx\frac{1}{(\sigma-1/2)^{\frac{\alpha+2}{p}}}\qquad\text{when $\sigma\to\frac12^
	{+}\,\cdot$}
\end{equation}
where the underlying constants depend on $p$ and $\alpha$ (and $a$ in \eqref{encadrinfini}) only.

\end{theorem}

\begin{proof}
It suffices to prove the lower estimates since the upper ones are known from Theorem~\ref{estimacionconocida}.

Let us show the lower estimate in \eqref{encadrinfini}.

To this purpose, let $P$ be a Dirichlet polynomial vanishing at infinity. Then, there exists a Dirichlet series $f$, vanishing also at infinity, such that $I_t(f)=P$. Taking this into account and applying Theorem \ref{equivnormas}, we have that

$$\|f\|_{\mathcal{A}^p_{\alpha+tp}}\approx\|I_tf\|_{\mathcal{A}^p_{\alpha}}=\|P\|_{\mathcal{A}^p_{\alpha}}$$
where the constants depend on $\alpha$, $p$, and $t$ only. 

We already have the result when $\alpha\in(-1,0)$ thanks to Theorem~\ref{estiminf}. Let $\alpha\geq0$ and $\alpha_0\in(-1,0)$. 
Choose $t>0$ so that $\alpha=\alpha_0+tp$.
This, together with the boundedness of the functional $\delta_s$, $s\in\C_{1/2}$, on $\mathcal{A}^p_{\alpha,\infty}$, $\alpha>0$, yields  
\begin{align*}
|P(s)|= |I_t(f)(s)|=\left|\int_0^{+\infty}\frac{u^{t-1}}{\Gamma(t)}f(u+s)\,du\right|
&\leq
\int_0^{+\infty}\frac{u^{t-1}}{\Gamma(t)}\|\delta_{u+s}\|_{(\mathcal{A}^p_{\alpha_0+tp,\infty})^*}
\|f\|_{\mathcal{A}^p_{\alpha_0+tp}}\,du\\
&\lesssim
\int_0^{+\infty}\frac{u^{t-1}}{\Gamma(t)}\|\delta_{u+s}\|_{(\mathcal{A}^p_{\alpha,\infty})^*}
\|P\|_{\mathcal{A}^p_{\alpha_0}}\,du.
\end{align*}
That is,
\begin{equation}\label{majodelta}
\Gamma(t)\|\delta_s\|_{(\mathcal{A}^p_{\alpha_0,\infty})^*}
\leq 
\int_0^{+\infty} u^{t-1}\|\delta_{u+s}\|_{(\mathcal{A}^p_{\alpha,\infty})^*}\,du.
\end{equation}
We can assume that $s$ is real. 
We pick some positive $\lambda$ (we shall choose it large enough later) and split the latter integral in two and apply Lemma \ref{decreasing} to the first term and Theorem \ref{estimacionconocida} to the second one so that, for some $c>0$ (coming from Th.\ref{estimacionconocida}),
\begin{align*}
\int_0^{\lambda(s-\frac12)}& u^{t-1}\|\delta_{u+s}\|_{(\mathcal{A}^p_{\alpha,\infty})^*}\,du
+
\int_{\lambda(s-\frac12)}^{+\infty}
u^{t-1}\|\delta_{u+s}\|_{(\mathcal{A}^p_{\alpha,\infty})^*}\,du\\
&\leq 
\|\delta_{s}\|_{(\mathcal{A}^p_{\alpha,\infty})^*} \int_0^{\lambda(s-\frac12)} u^{t-1}\,du
+ c
\int_{\lambda(s-\frac12)}^{+\infty}\frac{u^{t-1}}{(u+s-1/2)^{\frac{\alpha+2}{p}}}\,du\\
&\leq
\dfrac{\lambda^t}{t}(s-1/2)^t\|\delta_{s}\|_{(\mathcal{A}^p_{\alpha,\infty})^*}+c\int_{\lambda(s-\frac12)}^{+\infty}
\frac{du}{u^{\frac{\alpha_0+2}{p}+1}}\,du\\&
\leq 
\dfrac{\lambda^t}{t}(s-1/2)^t\|\delta_{s}\|_{(\mathcal{A}^p_{\alpha,\infty})^*}
+
\frac{K(\lambda)}{(s-1/2)^{\frac{\alpha_0+2}{p}}}
\end{align*}
where $K(\lambda)=\frac{pc}{\alpha_0+2}\lambda^{-\frac{\alpha_0+2}{p}}$, $K(\lambda)\to0$ as $\lambda\to\infty$. 

By Theorem \ref{estiminf} and using \eqref{majodelta}, we get
\[
\frac{c_0\Gamma(t)}{(s-1/2)^{\frac{\alpha_0+2}{p}}}\leq\Gamma(t)\|\delta_s\|_{(\mathcal{A}^p_{\alpha_0,\infty})^*}
\leq\dfrac{\lambda^t}{t}(s-1/2)^t\|\delta_{s}\|_{(\mathcal{A}^p_{\alpha,\infty})^*}
+
\frac{K(\lambda)}{(s-1/2)^{\frac{\alpha_0+2}{p}}}\cdot
\]
where $c_0>0$.

Now we choose  $\lambda$ large enough so that $K(\lambda)<\frac12c_0\Gamma(t)$, which gives,
\[
\frac{1}{(s-1/2)^\frac{\alpha+2}{p}}
=
\frac{1}{(s-1/2)^{\frac{\alpha_0+2}{p}+t}}\lesssim\|\delta_{s}\|_{(\mathcal{A}^p_{\alpha,\infty})^*}. 
\]

Thanks to Lemma~\ref{comparnormpointeval}, we clearly obtain \eqref{mainapprox}.

Finally, to prove the lower estimate in \eqref{encadr}, testing the constant functions, which we are not allowed for $\mathcal{A}^{p}_{\alpha,\infty}$, we first point out that $\|\delta_{s}\|_{(\mathcal{A}^p_{\alpha})^*}\geq1$ for every $s\in\C_{\frac{1}{2}}$.
Therefore\vskip-5pt
$$\|\delta_{s}\|_{(\mathcal{A}^p_{\alpha})^*}\geq\max\Big\{1;\|\delta_{s}\|_{(\mathcal{A}^p_{\alpha,\infty})^*}\Big\}\gtrsim\Big(\frac{\sigma}{\sigma-1/2}\Big)^{\frac{\alpha+2}{p}}\;.$$
\vskip-0.55cm\end{proof}

The third-named author was supported in part by the Labex CEMPI (ANR-11-LABX-0007-01).
\vskip-10pt
\bibliographystyle{siam}
\bibliography{biblioSGRLD.bib}
\end{document}